\documentclass[12pt]{amsart}
\usepackage{amsfonts}
\usepackage{mathrsfs}
\usepackage{amsmath,amssymb,amsthm,}
\numberwithin{equation}{section}

\def\N{\mathbb{N}{\ssc\,}}
\def\Z{\mathbb{Z}{\ssc\,}}

\def\R{\mathbb{R}{\ssc\,}}
\def\C{\mathbb{C}{\ssc\,}}

\def\ssc{\scriptscriptstyle}

\def \SVir{{\rm{SVir}}}
\def \Vir{{\rm{Vir}}}

\def \W{{\mathcal W}}
\def \R{{\mathcal R}}
\def \V{{\mathcal{V}}}
\def \NS{{\mathcal {NS}}}
\def \hh{{\mathfrak h}}

\def \<{\langle}
\def \>{\rangle}

\def\vs{\vspace*}

\def \be{\begin{equation}\label}
\def \ee{\end{equation}}
\def \bex{\begin{example}\label}
\def \eex{\end{example}}
\def \bl{\begin{lem}\label}
\def \el{\end{lem}}
\def \bt{\begin{thm}\label}
\def \et{\end{thm}}
\def \bp{\begin{prop}\label}
\def \ep{\end{prop}}
\def \br{\begin{rem}\label}
\def \er{\end{rem}}
\def \bc{\begin{coro}\label}
\def \ec{\end{coro}}
\def \bd{\begin{de}\label}
\def \ed{\end{de}}

\newtheorem{thm}{Theorem}[section]
\newtheorem{prop}[thm]{Proposition}
\newtheorem{coro}[thm]{Corollary}

\newtheorem{example}[thm]{Example}
\newtheorem{lem}[thm]{Lemma}

\newtheorem{rem}[thm]{Remark}
\newtheorem{de}[thm]{Definition}

\def\adddot{$\!\!\!${\bf}\ \ }
\makeatletter \@addtoreset{equation}{section}

\makeatother \makeatletter
\begin{document}

\title{A family of non-weight modules over the super-Virasoro algebras }
\author{Hengyun Yang}
\address{\it Hengyun Yang: Department of Mathematics, Shanghai Maritime     University, Shanghai 201306, China}
\email{hyyang@shmtu.edu.cn}

\author{Yufeng Yao}
\address{\it Yufeng Yao: Department of Mathematics, Shanghai Maritime   University, Shanghai 201306, China}
\email{yfyao@shmtu.edu.cn}

\author{Limeng Xia$^*$}
\address{\it Limeng Xia:  Institute of Applied System Analysis, Jiangsu University, Jiangsu Zhenjiang, 212013,
China}\email{xialimeng@ujs.edu.cn}
\thanks{$^*$The correspongding author}
\thanks{\rm This work is supported by the National Natural Science Foundation of China (Grant Nos. 11771279, 11771142, 11871249,11801363, 11671138 and 11571008), the Natural Science Foundation of Jiangsu (Grant No. BK20171294) and the Natural Science Foundation of Shanghai (Grant No. 16ZR1415000).}

\begin{abstract} In this paper, we construct a  family of non-weight modules over the  super-Virasoro algebras. Those modules when regarded as modules of the Ramond  algebra and further restricted as modules over the Cartan subalgebra $\hh$ are free of rank $1$, while when regarded as modules of the Neveu-Schwarz  algebra and further restricted as modules over the Cartan subalgebra $\mathfrak{H}$ are free of rank $2$. We obtain a sufficient and necessary condition for such modules to be simple. Moreover, we determine the isomorphism classes of these modules. Finally, we show that these modules constitute a complete classification of free $U(\hh)$-modules of rank $1$ over the  super-Virasoro algebra of Ramond type, and also constitute a complete classification of free $U(\mathfrak{H})$-modules of rank $2$ over the super-Virasoro algebra of Neveu-Schwarz type.
\end{abstract}

\maketitle


\section{Introduction}

The Virasoro algebra, as a universal central extension of the Witt algebra, is an important infinite dimensional Lie algebra. It has a basis $\{L_{n},C \mid n \in \Z\}$ subject to the following relations
\begin{eqnarray*}
    &&[L_{m}, L_{n}]=(m-n)L_{m+n}+\delta_{m+n}\frac{m^3-m}{12}C,\\
    &&[L_m,C]=0, \ \ \ \ \ \ m,n\in\Z.
\end{eqnarray*}
The super-Virasoro
algebras, namely, the Neveu-Schwarz algebra and the Ramond algebra, are the $N=1$ super-symmetry extensins of the Virasoro algebra (see \cite{NS} and \cite{RP}). 
The Virasoro algebra and super-Virasoro
algebras are important algebras in both mathematics and physics. They are closely related to the conformal field theory and the string theory.

In the study of representation theory of Lie algebras and Lie
superalgebras, one of the most important problems is to classify
simple modules. The simple weight modules of the Virasoro algebra
and super-Virasoro algebras have been studied extensively. In
1992, Mathieu \cite{M} (also see \cite{KS, Su1}) classified all
simple Harish-Chandra modules over the Virasoro algebra, i.e.,
simple weight modules with finite-dimensional weight spaces. Su
\cite{Su} settled similar problems over the super-Virasoro
algebras. While the study of simple non-weight modules is much
more difficult.

Free $U(\hh)$-modules constitute an important class of non-weight modules (in general, we use $U(\hh)$ to denote the universal enveloping algebra of the Cartan subalgebra $\hh$). They were first introduced by Nilsson \cite{N1} for the complex matrices algebra $\mathfrak{sl}_{n+1}$ in 2015. For a finite dimensional simple Lie algebra, Nilsson \cite{N2} showed that free $U(\hh)$-modules can exist only when it is of type $A$ or $C$, and further classified the free $U(\hh)$-modules of rank $1$. Motivated by \cite{N1}, Tan and Zhao \cite{TZ} proved that any free $U(\hh)$-module of rank 1  over the Witt algebra is isomorphism to $\Omega (\lambda,\alpha)$ for some $\lambda\in\C^*, \alpha\in\C$. With aid of the method  of \cite{N1} and the results of \cite{TZ}, such  simple modules for some  infinite dimensional Lie algebras related to the Virasoro algebra, such as the Heisenberg-Virasoro algebra \cite{CG}, the algebra $\Vir(a,b)$ \cite{HCS}, the Schrodinger-Virasoro algebra \cite{WZ} and the Block algebra \cite{CY, LG},  have been studied. Moreover, in \cite{MP},  a family of free $U(\hh)$-modules of arbitrary rank were constructed over $\mathfrak{sl}_2$.  The free $U(\hh)$-modules of infinite rank over the Virasoro algebra were studied in \cite{CG1}.

In \cite{LPX1}, Whittaker modules over the super-Virasoro algebras are introduced, and
necessary and sufficient conditions for irreducibility of these modules are given.  The authors further studied the simple restricted  modules over the super-Virasoro algebra of Neveu-Schwarz type in \cite{LPX2}.
Recently, free $U(\hh)$-modules over the basic Lie superalgebras have been studied in \cite{CZ}. It was shown that
$\mathfrak{osp}(1|2n)$ is the only basic Lie superalgebra that admits such modules.

For convenience, we let $\hh$ and $\mathfrak{H}$ denote the Cartan
subalgebras of $\SVir[0]$ and $\SVir[\frac12]$, respectively.  In
this paper, we mainly concern on the non-weight representations
over the super-Virasoro algebra $\SVir[\epsilon]$ ($\epsilon=0\,
\mbox{or} \, \frac{1}{2}$).
More precisely, we will classify the free  $U(\hh)$-modules of rank $1$ over $\SVir[0]$ and the free  $U(\mathfrak{H})$-modules of rank $2$ over $\SVir[\frac12]$. 

This paper is organized as follows. In section 2, we mainly recall some basic definitions and construct a family of non-weight modules over the super-Virasoro algebras of Ramond type $\SVir[0]$ and Neveu-Schwarz type $\SVir[\frac{1}{2}]$. We further determine the simplicity and isomorphism classes of these modules. Section 3 is devoted to classifying the free $U(\hh)$-modules of rank $1$ over the  Ramond algebra $\SVir[0]$. Those modules can also be regarded as free $U(\C{L_0})$-modules of rank $2$. In Section 4, we classify the free $U(\mathfrak{H})$-modules of rank $2$ over the Neveu-Schwarz algebra $\SVir[\frac12]$. We show that the category of free $U(\hh)$-modules of rank $1$ over the super-Virasoro algebra of Ramond type is equivalent to the category of free $U(\mathfrak{H})$-modules of rank $2$ over the super-Virasoro algebra of Neveu-Schwarz type.

\section{A family of non-weight modules over the super-Virasoro algebras}

Throughout the paper, we denote by $\C,\C^*,\Z, \Z_+$ and $\N$ the sets of all complex numbers, nonzero complex numbers, integers, non-negative integers and positive integers, respectively. We always assume that the base field is the complex number field $\C$. All vector superspaces (resp. superalgebras, supermodules) $V=V_{\bar 0}\oplus V_{\bar 1}$ are defined over $\C$, and sometimes simply called spaces (resp. algebras, modules). We call elements in $V_{\bar 0}$ and $V_{\bar 1}$ odd and even, respectively. Both odd and even elements are referred to homogeneous ones.
Throughout this paper, a module $M$ of a superalgebra $A$ always means a supermodule, i.e.,  $A_{\bar{i}}\cdot M_{\bar{j}}\subseteq M_{\bar{i}+\bar{j}}$ for all $\bar{i}, \bar{j}\in\Z_2$. For an automorphism $\sigma$ of a superalgebra $A$ and an $A$-module $M$, we have a twist $A$-module with the same underling space as $M$ and the twist module structure defined as $a\circ v=\sigma(a)\cdot v$ for $a\in A, v\in M$. This twist module is usually denoted by $M^{\sigma}$.
There is a parity change functor $\Pi$ from the category of $A$-modules to itself. That is, for any module $M=M_{\bar{0}}\oplus M_{\bar{1}}$, we have a new  module $\Pi(M)$ with the same underlining space with the parity exchanged, i.e., $(\Pi(M))_{\bar{0}}=M_{\bar{1}}$ and $(\Pi(M))_{\bar{1}}=M_{\bar{0}}$.

The following notion of the super-Virasoro algebra is a natural super generalization of the Virasoro algebra, which was first discovered
by Neveu-Schwarz and Ramond, respectively (see \cite{NS} and \cite{RP}).
\bd{dcobo}\label{def of super-Vir}
{\em
    Let $\epsilon =0\, \mbox{or} \, \frac{1}{2}$. \emph{The super-Virasoro algebra $\SVir[\epsilon]$ (without center)}, is an infinite dimensional Lie
superalgebra whose even part is spanned by $\{L_{n} \mid n \in \Z\}$ and odd part is spanned by $\{G_{r} \mid r \in \epsilon +\Z \}$ subject to the following relations
\begin{eqnarray*}
    &&[L_{m}, L_{n}]=(m-n)L_{m+n},\\
    &&[L_m,G_r]=(\frac{1}{2}m-r)G_{m+r},\\
    &&[G_r,G_s]=2L_{r+s}, \ \ \ \ \ \ m,n\in\Z, r,s\in \epsilon+\Z.
\end{eqnarray*}
$\SVir[0]:=\R$ is called \emph{the Ramond algebra} and $\SVir[\frac12]:=\NS$ is
called \emph{the Neveu-Schwarz algebra}.
}\ed

\begin{rem}
\begin{itemize}
\item[(1)] The even part $\SVir[\epsilon]_{\bar 0}$ of the super-Virasoro algebra $\SVir[\epsilon]$ is isomorphic to the centerless Virasoro algebra, i.e., the Witt algebra $\W$.

\item[(2)] For $\lambda\in\mathbb{C}^*$, let $\sigma_{\lambda}:\SVir[\epsilon]\longrightarrow\SVir[\epsilon]$ be a linear transformation with $\sigma_{\lambda}(L_m)=\lambda^m L_m$ and $\sigma_{\lambda}(G_r)=\lambda^rG_r$, where $m\in\mathbb{Z}, r\in\epsilon+\Z$. Then $\sigma_{\lambda}$ is an automorphism of $\SVir[\epsilon]$.
\end{itemize}
\end{rem}

The Ramond algebra $\mathcal{R}$ has a $2$-dimensional canonical Cartan subalgebra $\hh =\C L_0\oplus\C G_0$. While the  Neveu-Schwarz algebra $\NS$ has a $1$-dimensional canonical Cartan subalgebra $\mathfrak{H} =\C L_0$.

Denote by $\C[x]$ the polynomial algebra over $\C$ in  indeterminate $x$. For $\lambda\in\C^*, \alpha\in\C$, we define the action of $\W$ on $\Omega (\lambda,\alpha):=\C[x] $ as follows
\begin{eqnarray}\label{Virasoro-module}
L_m f(x)=\lambda^m(x+m\alpha)f(x+m),  \ \ \ \ \ \ m\in\Z, f(x)\in\C[x].
\end{eqnarray}

For later use, we need the following results on $\Omega (\lambda,\alpha)$. 

\begin{lem}\adddot\label{Witt-module}(cf. \cite[\S4.3]{LZ1}, \cite[Theorem 8, Theorem 11]{TZ})
Keep notations as above, then the following statements on $\Omega (\lambda,\alpha)$ hold.
\begin{itemize}
\item[\rm(1)]
    $ \Omega(\lambda,\alpha)$ is a $\W$-module.
\item[\rm(2)] $\Omega(\lambda,\alpha)$ is simple if and only if
$\lambda,\alpha\in\C^*$. Moreover, $x\Omega(\lambda,0)$ is the
unique nonzero proper submodule of $\Omega(\lambda,0)$ with
codimension $1$. \item[\rm(3)] Any $\W$-module that is free of
rank $1$ when restricted to  $U(\C L_0)$ is isomorphic to some
$\Omega(\lambda,\alpha)$ for $\lambda\in\C^*, \alpha\in\C$.
\end{itemize}
\end{lem}

Let $V =\C [x^2]\oplus x\C [x^2]$. Then $ V$  is a $\Z_2$-graded vector space with $V_{\bar{0}}=\C [x^2]$ and $V_{\bar{1}}=x\C[x^2]$. The following result gives a precise construction of an $\R$-module structure on $V$.

\begin{prop}\adddot
    \label{prop1} For $\lambda\in\C^*,\alpha\in\C,  f(x^2)\in \C [x^2]$ and $xf(x^2)\in x\C [x^2]$, we define the action of $\R$ on $V$ as follows
    \begin{eqnarray}
    &&L_m f(x^2)=\lambda^m(x^2+m\alpha)f(x^2+m),\label{module1}\\
    &&L_m xf(x^2)=\lambda^m x (x^2+m\alpha+\frac {m}{2})f(x^2+m), \label{module2}\\
    &&G_m f(x^2)=\lambda^m x f(x^2+m) ,\label{module3}\\
    &&G_m xf(x^2)=\lambda^m  (x^2+2m\alpha)f(x^2+m) , \label{module4}
    \end{eqnarray}
where $m\in\Z$. Then $V$ is an $\R$-module under the action of (\ref{module1})-(\ref{module4}), which is denoted by $\Omega_{\R}(\lambda,\alpha)$.
\end{prop}

\begin{proof}
It follows from Lemma \ref{Witt-module} (1) that
\begin{equation*}\label{check1}
[L_m, L_n] f(x^2)=L_mL_nf(x^2)-L_nL_mf(x^2), \,\,\forall\,m,n\in\Z.
\end{equation*}
According to (\ref{module2}), we have
\begin{eqnarray*}
L_m L_n xf(x^2)\!\!\!&=\!\!\!&\lambda^nL_m x(x^2+n\alpha+\frac {n}{2})f(x^2+n)\nonumber\\
\!\!\!&=\!\!\!&\lambda^{m+n}x(x^2+m\alpha+\frac {m}{2})(x^2+n\alpha+\frac {n}{2}+m)f(x^2+m+n),
\end{eqnarray*}
which implies that
\begin{eqnarray*}
&&L_mL_nxf(x^2)-L_nL_mxf(x^2)\\
\!\!\!&=\!\!\!&\lambda^{m+n}x\big((x^2+m\alpha+\frac {m}{2})(x^2+n\alpha+\frac {n}{2}+m)\\
&&-(x^2+m\alpha+\frac {m}{2}+n)(x^2+n\alpha+\frac {n}{2})\big)f(x^2+m+n)\\
\!\!\!&=\!\!\!&\lambda^{m+n}(m-n)x\big(x^2+(m+n)\alpha+\frac {m+n}{2}\big)f(x^2+m+n)\\
\!\!\!&=\!\!\!&(m-n)L_{m+n} xf(x^2),\\
\!\!\!&=\!\!\!&[L_m,L_n] xf(x^2).
\end{eqnarray*}
Moreover, it follows from (\ref{module1})-(\ref{module3}) that
\begin{eqnarray*}
    &&L_mG_nf(x^2)-G_nL_mf(x^2)\\
    \!\!\!&=\!\!\!&\lambda^nL_m xf(x^2+n)-\lambda^mG_n (x^2+m\alpha)f(x^2+m)\\
    \!\!\!&=\!\!\!&\lambda^{m+n}x(x^2+m\alpha+\frac {m}{2})f(x^2+m+n)-\lambda^{m+n}x(x^2+m\alpha+n)f(x^2+m+n)\\
    \!\!\!&=\!\!\!&(\frac{m}{2}-n)G_{m+n} f(x^2)\\
    \!\!\!&=\!\!\!&[L_m,G_n] f(x^2).
\end{eqnarray*}
From (\ref{module1}), (\ref{module2}) and (\ref{module4}), we obtain that
\begin{eqnarray*}
    &&L_mG_nxf(x^2)-G_nL_mxf(x^2)\\
    \!\!\!&=\!\!\!&\lambda^nL_m (x^2+2n\alpha)f(x^2+n)-\lambda^mG_n x(x^2+m\alpha+\frac{m}{2})f(x^2+m)\\
    \!\!\!&=\!\!\!&\lambda^{m+n}\big((x^2+m\alpha)(x^2+m+2n\alpha)-(x^2+2n\alpha)(x^2+m\alpha+\frac{m}{2}+n)\big)f(x^2+m+n)\\
    \!\!\!&=\!\!\!&\lambda^{m+n}(\frac{m}{2}-n)\big(x^2+2(m+n)\alpha\big) f(x^2+m+n)\\
    \!\!\!&=\!\!\!&(\frac{m}{2}-n)G_{m+n} xf(x^2)\\
    \!\!\!&=\!\!\!&[L_m,G_n] xf(x^2).
\end{eqnarray*}
Furthermore, it follows from (\ref{module1}), (\ref{module3}) and (\ref{module4}) that
\begin{eqnarray*}
    &&G_mG_nf(x^2)+G_nG_mf(x^2)\\
    \!\!\!&=\!\!\!&\lambda^{m+n}(x^2+2m\alpha)f(x^2+m+n)+\lambda^{m+n}(x^2+2n\alpha)f(x^2+m+n)\\
    \!\!\!&=\!\!\!&2\lambda^{m+n}(x^2+(m+n)\alpha) f(x^2+m+n)\\
    \!\!\!&=\!\!\!&2L_{m+n} f(x^2)\\
    \!\!\!&=\!\!\!&[G_m,G_n] f(x^2).
\end{eqnarray*}
By (\ref{module2}), (\ref{module3}) and (\ref{module4}), we get
\begin{eqnarray*}
    &&G_mG_nxf(x^2)+G_nG_mxf(x^2)\\
    \!\!\!&=\!\!\!&\lambda^{m+n}x(x^2+2n\alpha+m)f(x^2+m+n)+\lambda^{m+n}x(x^2+2m\alpha+n)f(x^2+m+n)\\
    \!\!\!&=\!\!\!&2\lambda^{m+n}(x^2+(m+n)\alpha) f(x^2+m+n)\\
    \!\!\!&=\!\!\!&2L_{m+n} xf(x^2)\\
    \!\!\!&=\!\!\!&[G_m,G_n] xf(x^2).
\end{eqnarray*}
This completes the proof.
\end{proof}

\begin{rem}
Let $\lambda\in\C^*, \alpha\in\mathbb{C}$, then $\Omega_{\R}(\lambda,\alpha)=\Omega_{\R}(1,\alpha)^{\sigma_{\lambda}}$.
\end{rem}

In the following, we exploit the module structure of
$\Omega_{\R}(\lambda,\alpha)$  and provide a sufficient and
necessary condition for simplicity of
$\Omega_{\R}(\lambda,\alpha)$.
\begin{thm}\label{thm-1}
Let $\R$ be the  Ramond agebra, $\lambda\in\C^*$ and $\alpha\in\C$. Let
$$\Xi:=x^2(\Omega_{\R}(\lambda,0)_{\bar{0}})\oplus \Omega_{\R}(\lambda,0)_{\bar{1}}.$$
Then the following statements hold.
\begin{itemize}
\item[\rm(1)]
 $\Omega_{\R}(\lambda,\alpha)$ is simple if and only if $\alpha\neq 0$.
\item[\rm(2)]
 $\Omega_{\R}(\lambda,0)$ has a unique proper submodule $\Xi$, and $\Omega_{\R}(\lambda,0)/\Xi$ is a $1$-dimensional trivial $\R$-module.
\item[\rm(3)]
$\Xi\cong\Pi(\Omega_{\R}(\lambda,\frac12))$, so that $\Xi$ is an irreducible $\R$-module.
\end{itemize}
\end{thm}

\begin{proof}
Let $M=M_{\bar 0}\oplus M_{\bar 1}$ be a nonzero submodule of
$\Omega(\lambda,\alpha)$. It follows from (\ref{module4}) that
$M_{\bar 0}\neq 0$. Moreover, if $M_{\bar
0}=\Omega(\lambda,\alpha)_{\bar 0}$, we then obtain that $M_{\bar
1}=\Omega(\lambda,\alpha)_{\bar 1}$ by  (\ref{module3}).
Consequently, $M=\Omega(\lambda,\alpha)$. This implies that
$\Omega(\lambda,\alpha)$ is a simple $\R$-module provided that
$\Omega(\lambda,\alpha)_{\bar 0}$ is a simple $\R_{\bar
0}$-module.

Case (i): $\alpha=0$.

In this case, it is readily shown that $\Xi$ is a proper
submodule. Furthermore, suppose $M=M_{\bar 0}\oplus M_{\bar 1}$ is
an arbitrary nonzero proper submodule of $\Omega(\lambda,\alpha)$,
then $M_{\bar 0}=x^2\Omega(\lambda,\alpha)_{\bar 0}$ by Lemma
\ref{Witt-module} (2). Furthermore, it follows from
(\ref{module3}) that $M_{\bar 1}=\Omega(\lambda,\alpha)_{\bar 1}$.
Consequently, $M=\Xi$, and $\Omega_{\R}(\lambda,0)/\Xi$ is a
$1$-dimensional trivial $\R$-module.

Case (ii): $\alpha\neq 0$.

In this case, since  $\Omega(\lambda,\alpha)_{\bar 0}$ is a simple $\R_{\bar 0}$-module by Lemma \ref{Witt-module} (2), it follows that  $\Omega(\lambda,\alpha)$ is a simple $\R$-module.

For the part (3), we define a linear map $\varphi$ from $\Xi $ to $\Pi(\Omega_{\R}(\lambda,\frac12))$ such that
$$\varphi(x^2f(x^2))=xf(x^2), \ \ \ \ \varphi(xf(x^2))=f(x^2) $$
for $f(x^2)\in\C[x^2]$. From (\ref{module1}) with $\alpha=0$ and (\ref{module2}) with $\alpha=\frac12$, we have
\begin{eqnarray*}
    \varphi(L_m x^2f(x^2))\!\!\!&=\!\!\!&\lambda^m\varphi(x^2(x^2+m)f(x^2+m))\\
    \!\!\!&=\!\!\!&\lambda^mx(x^2+m)f(x^2+m)\\
    \!\!\!&=\!\!\!&L_m xf(x^2)\\
    \!\!\!&=\!\!\!&L_m\varphi(x^2f(x^2)).
\end{eqnarray*}
Similar arguments yield that
\begin{eqnarray*}
    &&\varphi(G_m x^2f(x^2))=G_m\varphi(x^2f(x^2)),\\
    &&\varphi(L_m xf(x^2))=L_m\varphi(xf(x^2)),\\
    &&\varphi(G_m xf(x^2))=G_m\varphi(xf(x^2)).
\end{eqnarray*}
Thus $\varphi$ is an $\R$-module isomorphism, and $\Xi$ is irreducible.

We complete the proof.
\end{proof}

\begin{thm}\label{thm-2}
Let $\lambda,\mu\in\C^*, \alpha,\beta\in\C$. Then $\Omega_{\R}(\lambda,\alpha)\cong\Omega_{\R}(\mu,\beta)$ as $\R$-modules if and only if $\lambda=\mu$ and $\alpha=\beta$.
\end{thm}

\begin{proof} The sufficiency is obvious.
Suppose  $\Omega_{\R}(\lambda,\alpha)\cong\Omega_{\R}(\mu,\beta)$ as $\R$-modules. Then $\Omega_{\R}(\lambda,\alpha)_{\bar 0}\cong\Omega_{\R}(\mu,\beta)_{\bar 0}$ as $\R_{\bar 0}$-module. It follows that $\lambda=\mu$ and $\alpha=\beta$.

\end{proof}

Let $\V =\C [x]\oplus \C [y]$. Then $\V$  is a $\Z_2$-graded vector space with $\V_{\bar{0}}=\C [x]$ and $\V_{\bar{1}}=\C [y]$. The following result gives a precise construction of an $\NS$-module structure on $\V$.

\begin{prop}\adddot
    \label{prop2} For $\lambda\in\C^*,\alpha\in\C,  f(x)\in \C [x]$ and $g(y)\in \C [y]$, we define the action of $\NS$ on $\V$ as follows
    \begin{eqnarray}
    &&L_m f(x)=\lambda^m(x+m\alpha)f(x+m),\label{ns-module1}\\
    &&L_m g(y)=\lambda^m(y+m(\alpha+\frac{1}{2}))g(y+m),\label{ns-module2}\\
    &&G_r f(x)=\lambda^{r-\frac{1}{2}}f(y+r) ,\label{ns-module3}\\
    &&G_r g(y)=\lambda^{r+\frac{1}{2}}(x+2r\alpha)g(x+r) ,\label{ns-module4}
    \end{eqnarray}
where $m\in\Z, r\in\frac{1}{2}+\Z$. Then $\V$ is an $\NS$-module under the action of (\ref{ns-module1})-(\ref{ns-module4}), which is denoted by $\Omega_{\NS}(\lambda,\alpha)$.
\end{prop}

\begin{proof}
It follows from Lemma \ref{Witt-module} (1) that
\begin{eqnarray*}
&&[L_m, L_n] f(x)=L_mL_nf(x)-L_nL_mf(x),\label{check-ns1}\\
&&[L_m, L_n]
g(y)=L_mL_ng(y)-L_nL_mg(y),
\,\,\forall\,m,n\in\Z.\label{check-ns2}
\end{eqnarray*}
Moreover, due to (\ref{ns-module1}), (\ref{ns-module2}) and
(\ref{ns-module3}), we obtain
\begin{eqnarray*}
    &&L_mG_rf(x)-G_rL_mf(x)\\
    \!\!\!&=\!\!\!&\lambda^{m+r-\frac{1}{2}}(y+m(\alpha+\frac{1}{2}))f(y+m+r)-\lambda^{m+r-\frac{1}{2}}(y+r+m\alpha)f(y+m+r)\\
    \!\!\!&=\!\!\!&(\frac{m}{2}-r)\lambda^{m+r-\frac{1}{2}}f(y+m+r)\\
    \!\!\!&=\!\!\!&(\frac{m}{2}-r)G_{m+r} f(x)\\
    \!\!\!&=\!\!\!&[L_m,G_r] f(x).
\end{eqnarray*}
Similarly, Using (\ref{ns-module1}), (\ref{ns-module2}) and
(\ref{ns-module4}), we know that
\begin{eqnarray*}
    &&L_mG_rg(y)-G_rL_mg(y)\\
    \!\!\!&=\!\!\!&\lambda^{m+r+\frac{1}{2}}(x+m\alpha)(x+m+2r\alpha)g(x+m+r)\\
    &&-\lambda^{m+r+\frac{1}{2}}(x+2r\alpha)(x+r+m(\alpha+\frac{1}{2}))g(x+m+r)\\
    \!\!\!&=\!\!\!&(\frac{m}{2}-r)\lambda^{m+r+\frac{1}{2}}(x+2(m+r)\alpha)g(x+m+r)\\
    \!\!\!&=\!\!\!&(\frac{m}{2}-r)G_{m+r} g(y)\\
    \!\!\!&=\!\!\!&[L_m,G_r] g(y).
\end{eqnarray*}
Furthermore, from (\ref{ns-module1}), (\ref{ns-module3}) and
(\ref{ns-module4}), we have
\begin{eqnarray*}
    &&G_rG_sf(x)+G_sG_rf(x)\\
    \!\!\!&=\!\!\!&\lambda^{r+s}(x+2r\alpha)f(x+r+s)+\lambda^{r+s}(x+2s\alpha)f(x+r+s)\\
    \!\!\!&=\!\!\!&2\lambda^{r+s}(x+(r+s)\alpha)f(x+r+s)\\
    \!\!\!&=\!\!\!&2L_{r+s} f(x)\\
    \!\!\!&=\!\!\!&[G_r,G_s] f(x).
\end{eqnarray*}
Similarly, by (\ref{ns-module2}), (\ref{ns-module3}) and
(\ref{ns-module4}), we get
\begin{eqnarray*}
    &&G_rG_sg(y)+G_sG_rg(y)\\
    \!\!\!&=\!\!\!&\lambda^{r+s}(y+r+2s\alpha)g(y+r+s)+\lambda^{r+s}(y+s+2r\alpha)g(y+r+s)\\
    \!\!\!&=\!\!\!&2\lambda^{r+s}(y+(r+s)(\alpha+\frac{1}{2}))g(y+r+s)\\
    \!\!\!&=\!\!\!&2L_{r+s} g(y)\\
    \!\!\!&=\!\!\!&[G_r,G_s] g(y).
\end{eqnarray*}
Now we proved that $\V$ is an $\NS$-module.
\end{proof}

\begin{rem}
Let $\lambda\in\C^*, \alpha\in\mathbb{C}$, then $\Omega_{\NS}(\lambda,\alpha)=\Omega_{\NS}(1,\alpha)^{\sigma_{\lambda}}$.
\end{rem}

Let $\sigma:\NS \rightarrow\R$ be a linear map defined by
\begin{eqnarray*}
L_m \!\!\!&\mapsto\!\!\!&\frac{1}{2} L_{2m},\nonumber\\
G_{r}\!\!\!&\mapsto\!\!\!&\frac{1}{\sqrt{2}} G_{2r}
\end{eqnarray*}
for $m\in\Z, r\in\frac12+\Z$. It is straightforward to verify that $\sigma$ is an injective Lie superalgebra homomorphism. Then $\NS$ can be regarded as a subalgebra of $\R$. In particular, $\Omega_{\R}(\lambda,\alpha)$ is an $\NS$-module.

\begin{prop}\label{Connection of modules over R and NS}
$\Omega_{\R}(\lambda,\alpha)$ is a simple $\NS$-module if and only if it is a simple $\R$-module.
\end{prop}
\begin{proof}
It is sufficient to prove that $\Omega_{\R}(\lambda,\alpha)$ is a simple $\NS$-module whenever $\alpha\not=0$.
Indeed, since $\sigma(L_m)f(x^2)=\frac12\lambda^{2m}(x^2+2m\alpha)f(x^2+2m)$ for all $m\in\Z$, it follows that $\Omega_{\R}(\lambda,\alpha)_{\bar0}$ is a simple $\sigma(\W)$-module. Moreover, since
$$\sigma(G_{r})xf(x^2)=\frac{1}{\sqrt{2}}\lambda^{2r}(x^2+2r\alpha)f(x^2+2r)$$
and
$$\sigma(G_{r})f(x^2)=\frac{1}{\sqrt{2}}\lambda^{2r}xf(x^2+2r),$$
for all $r\in\frac12+\Z$, it further yields that $\Omega_{\R}(\lambda,\alpha)$ is a simple $\NS$-module.
\end{proof}

The following result gives an intimate connection between the two classes of modules constructed in Proposition \ref{prop1} and Proposition \ref{prop2}.
\begin{prop}\label{an iso}
Let $\lambda\in\C^*,\alpha\in\C$. Then as $\NS$-modules, $\Omega_{\NS}(\lambda,\alpha)\cong\Omega_{\R}(\sqrt{\lambda},\alpha)$.
\end{prop}
\begin{proof}
Let
\begin{eqnarray*}
\Phi:\,\Omega_{\NS}(\lambda,\alpha)& \longrightarrow &\Omega_{\R}(\sqrt{\lambda},\alpha)\\
f(x)&\longmapsto& f(\frac{1}{2}x^2)\\
g(y)&\longmapsto& \frac{\sqrt{\lambda}}{\sqrt{2}}x
g(\frac{1}{2}x^2)
\end{eqnarray*}
be the linear map  from $\Omega_{\NS}(\lambda,\alpha)$ to $\Omega_{\R}(\sqrt{\lambda},\alpha)$. It is obvious that $\Phi$ is bijective. In the following, we will show that it is an $\NS$-module isomorphism.

For any $m\in\mathbb{Z}, r\in\frac{1}{2}+\mathbb{Z}, f(x)\in\C[x],
g(y)\in\C[y]$, on the one hand,
\begin{eqnarray*}
&&\Phi(L_mf(x))=\Phi(\lambda^m(x+m\alpha)f(x+m))=\lambda^m(\frac{1}{2}x^2+m\alpha)f(\frac{1}{2}x^2+m),\\
&&\Phi(L_mg(y))=\Phi(\lambda^m(y+m(\alpha+\frac{1}{2}))g(y+m))=
\lambda^m\frac{\sqrt{\lambda}}{\sqrt{2}}x(\frac{1}{2}x^2+m(\alpha+\frac{1}{2}))g(\frac{1}{2}x^2+m),\\
&&\Phi(G_rf(x))=\Phi(\lambda^{r-\frac{1}{2}}f(y+r))=\frac{\lambda^{r}}{\sqrt{2}}xf(\frac{1}{2}x^2+r),\\
&&\Phi(G_rg(y))=\Phi(\lambda^{r+\frac{1}{2}}(x+2r\alpha)g(x+r))=\lambda^{r+\frac{1}{2}}(\frac{1}{2}x^2+2r\alpha)g(\frac{1}{2}x^2+r).
\end{eqnarray*}
While on the other hand,
\begin{eqnarray*}
&&\sigma(L_m)\Phi(f(x))=\frac{1}{2}L_{2m}f(\frac{1}{2}x^2)=\frac{1}{2}(\sqrt{\lambda})^{2m}(x^2+2m\alpha)f(\frac{1}{2}(x^2+2m)),\\
&&\sigma(L_m)\Phi(g(y))=\frac{1}{2}L_{2m}\frac{\sqrt{r}}{\sqrt{2}}xf(\frac{1}{2}x^2)=
\frac{\lambda^{m+\frac{1}{2}}}{2\sqrt{2}}x(x^2+2m\alpha+m)g(\frac{1}{2}(x^2+2m)),\\
&&\sigma(G_r)\Phi(f(x))=\frac{1}{\sqrt{2}}G_{2r}f(\frac{1}{2}x^2)=\frac{1}{\sqrt{2}}\lambda^rxf(\frac{1}{2}(x^2+2r)),\\
&&\sigma(G_r)\Phi(g(y))=\frac{1}{\sqrt{2}}G_{2r}\frac{\sqrt{\lambda}}{\sqrt{2}}xg(\frac{1}{2}x^2)
=\frac{\lambda^{r+\frac{1}{2}}}{2}(x^2+4r\alpha)g(\frac{1}{2}(x^2+2r)).
\end{eqnarray*}
Hence, we get
\begin{eqnarray*}
&&\Phi(L_mf(x))=\sigma(L_m)\Phi(f(x)),\,\,\,\,\,\,\Phi(L_mg(y))=\sigma(L_m)\Phi(g(y)),\\
&&\Phi(G_rf(x))=\sigma(G_r)\Phi(f(x)),\,\,\,\,\,\,\Phi(G_rg(y))=\sigma(G_r)\Phi(g(y)).
\end{eqnarray*}
This implies that $\Phi$  is an $\NS$-module isomorphism, as
desired.
\end{proof}

As a direct consequence of Proposition \ref{an iso}, Proposition \ref{Connection of modules over R and NS} and Theorem \ref{thm-1}, we have
\begin{coro}
Let $\NS$ be the  Neveu-Schwarz agebra, $\lambda\in\C^*$ and $\alpha\in\C$. Then $\Omega_{\NS}(\lambda,\alpha)$ is simple if and only if $\alpha\neq 0$.
\end{coro}

As a paralleling result to Theorem \ref{thm-1}, we have the following proposition describing the $\NS$-module structure of  $\Omega_{\NS}(\lambda,0)$ for $\lambda\in\C^*$.
\begin{prop}
Let $\lambda\in\C^*$ and $\Gamma=x(\Omega_{\NS}(\lambda,0)_{\bar{0}})\oplus \Omega_{\NS}(\lambda,0)_{\bar 1}$.
Then the following statements hold.
\begin{itemize}
\item[(1)] $\Omega_{\NS}(\lambda,0)$ has a unique proper submodule $\Gamma$, and $\Omega_{\NS}(\lambda,0)/\Gamma$ is a $1$-dimensional trivial $\NS$-module.
\item[(2)] $\Gamma\cong\Pi(\Omega_{\NS}(\lambda,\frac{1}{2}))$, so that $\Gamma$ is an irreducible $\NS$-module.
\end{itemize}
\end{prop}

\begin{proof}
(1) It is readily shown that $\Gamma$ is an $\NS$-submodule of $\Omega_{\NS}(\lambda,0)$ by (\ref{ns-module1})-(\ref{ns-module4}). Now suppose $M=M_{\bar{0}}\oplus M_{\bar{1}}$ is an arbitrary nonzero submodule of $\Omega_{\NS}(\lambda,0)$. It follows from (\ref{ns-module4}) that $M_{\bar 0}\neq 0$. Moreover, $M_{\bar 0}\neq \Omega_{\NS}(\lambda,0)_{\bar{0}}$ by (\ref{ns-module3}), since $M$ is a proper submodule. Hence, $M_{\bar{0}}=x(\Omega_{\NS}(\lambda,0)_{\bar{0}})$ by Lemma \ref{Witt-module} (2). Then it follows from (\ref{ns-module3}) that $M_{\bar 1}=\Omega_{\NS}(\lambda,0)_{\bar 1}$. Consequently, $M=\Gamma$, and $\Omega_{\NS}(\lambda,0)/\Gamma$ is a $1$-dimensional trivial $\NS$-module.

(2) Let $\psi$ be a linear map from $\Gamma$ to $\Pi(\Omega_{\NS}(\lambda,\frac{1}{2}))$ with $\psi(xf(x))=\frac1{\sqrt{2\lambda}}f(y)$ and $\psi(f(y))=\frac{\sqrt{\lambda}}{\sqrt{2}}f(x)$. Then it follows from a direct computation that $\psi$ is an $\NS$-module isomorphism. Hence, $\Gamma\cong\Pi(\Omega_{\NS}(\lambda,\frac{1}{2}))$, and $\Gamma$ is an irreducible $\NS$-module.
\end{proof}

Similar arguments as in the proof of Theorem \ref{thm-2} yield the following paralleling result for the super-Virasoro algebra of Neveu-Schwarz type.
\begin{thm}
Let $\lambda,\mu\in\C^*, \alpha,\beta\in\C$. Then  $\Omega_{\NS}(\lambda,\alpha)\cong\Omega_{\NS}(\mu,\beta)$ as $\NS$-modules if and only if $\lambda=\mu$ and $\alpha=\beta$.
\end{thm}

\section{Classification of free $U(\mathfrak{h})$-modules of rank $1$ over $\R$}

Let $M=M_{\bar 0}\oplus M_{\bar 1}$ be an $\R$-module such that it is free of rank $1$ as a $U(\mathfrak{h})$-module, where $\hh=\C L_0\oplus\C G_0$ is the canonical Cartan subalgebra of $\R$.   According to the algebra structure of $\R$ in Definition \ref{def of super-Vir}, we know that $L_0G_0=G_0L_0$ and $G_0^2=L_0$. Thus $U(\hh)=\C [L_0]\oplus G_0\C [L_0]$. Take a homogeneous basis element ${\bf1}\in M$, without loss of generality, up to a parity, we may assume ${\bf 1}\in M_{\bar 0}$, and  $$M=U(\hh){\bf1}=\C [L_0]{\bf1}\oplus G_0\C [L_0]{\bf1}$$ with $M_{\bar 0}=\C[L_0]{\bf1}$ and $M_{\bar 1}=G_0\C [L_0]{\bf1}$.

 Hence, we suppose that $L_0{\bf1}=x^2{\bf1}$ and $G_0{\bf1}=x{\bf1}$. In the following we identify $M$ with $\C[x^2]{\bf1}\oplus x\C[x^2]{\bf1}$ such that $M_{\bar 0}=\C[x^2]{\bf1}$ and $M_{\bar 1}= x\C[x^2]{\bf1}$.

\begin{rem}
The free $U(\mathfrak{h})$-module M  of rank $1$ can be regarded as a free $\C[L_0]$-module of rank $2$.
\end{rem}

Since the even part $\R_{\bar 0}$ of $\R$ is isomorphic to the Witt algebra $\W$, we can naturally regard $M_{\bar 0}$ as a $\W$-module which is free of rank $1$ as a $\C[L_0]$-module. By Lemma \ref{Witt-module}, for any $m\in\Z, f(x^2)\in\C[x^2]$, there exits  $\lambda\in\C^*$, $\alpha\in\C$ such that
\begin{eqnarray}\label{Lmf2}
L_m f(x^2){\bf1}=\lambda^m(x^2+m\alpha)f(x^2+m){\bf1}.
\end{eqnarray}

We further need the following two preliminary results for later use.

\begin{lem}\label{lemma-GLm}\adddot
For $m\in\Z$, $f(x^2)\in\C[x^2]$, we have
\begin{itemize}
\item[(1)] $G_mxf(x^2){\bf1}=f(x^2+m)G_m x{\bf1}$. \item[(2)]
$G_mf(x^2){\bf1}=f(x^2+m)G_m {\bf1}$.
\end{itemize}
\end{lem}
\begin{proof}
We only prove part (1). Similar arguments yield part (2).
According to the defining relations of $\R$, we have
\begin{eqnarray*}
G_mL_0 x{\bf1}=(L_0+m) G_m x{\bf1}.
\end{eqnarray*}
Then by induction on $n$, we obtain that
\begin{eqnarray*}\label{LG1}
G_mL_0^n x{\bf1}=(L_0+m)^n G_m x{\bf1},   \,\,\forall\, \ n\in\Z_+.
\end{eqnarray*}
Consequently,
\begin{eqnarray*}\label{LG2}
G_m xf(x^2){\bf1}= G_mf(L_0) x{\bf1}=f(L_0+m)G_m x{\bf1}=f(x^2+m)G_m x{\bf1},  \,\,\forall\, m\in\Z.
\end{eqnarray*}
\end{proof}

\begin{lem}\label{lemma-Gm}\adddot
For $m\in\Z$, we have
$G_m {\bf1}=\lambda^m x{\bf1}$.
\end{lem}
\begin{proof}
Suppose $G_m {\bf1}=xg_m(x^2){\bf1}\in x\C [x^2]{\bf1}$ for $m\in\Z$. Obviously it holds for $m=0$ since $G_0{\bf1}=x{\bf1}$. We first show that $G_m{\bf1}=\lambda^m x{\bf1}$ holds for $m\in\N$ by induction on $m$.  Taking $m=0$ in part (1) of Lemma \ref{lemma-GLm}, we have
\begin{eqnarray}\label{LG3}
G_0G_1 {\bf1}=G_0 xg_1(x^2){\bf1}=x^2g_1(x^2){\bf1}.
\end{eqnarray}
Then it follows from (\ref{Lmf2}), (\ref{LG3}) and $[G_0,G_1]=2L_1$  that
\begin{equation}\label{GG1}
G_1 x{\bf1}=G_1G_0 {\bf1}=2L_1 {\bf1}-G_0G_1 {\bf1}=(2\lambda(x^2+\alpha)- x^2g_1(x^2)){\bf1}.
\end{equation}
Taking $m=1$ in part (1) of Lemma \ref{lemma-GLm} and using (\ref{GG1}), we deduce that
\begin{eqnarray}\label{G1x}
G_1 xg_1(x^2){\bf1}= g_1(x^2+1)G_1 x{\bf1}= g_1(x^2+1)(2\lambda(x^2+\alpha)- x^2g_1(x^2)){\bf1}.
\end{eqnarray}
Having in mind that $2G_1G_1{\bf1}=[G_1,G_1]{\bf1}=2L_2{\bf1}$ and applying (\ref{Lmf2}), it follows that
$$G_1  xg_1(x^2){\bf1}=G_1^2{\bf1}=L_2{\bf1}=\lambda^2(x^2+2\alpha){\bf1},$$
which together with (\ref{G1x}) yields that
\begin{equation}\label{an eq}
g_1(x^2+1)(2\lambda(x^2+\alpha)- x^2g_1(x^2))=\lambda^2(x^2+2\alpha).
\end{equation}
By comparing the degrees of both polynomials in (\ref{an eq}), we obtain $g_1(x^2)=\lambda$. Thus $$G_1 {\bf1}=\lambda x{\bf1}.$$
Using this result and the relation $[L_1,G_0]{\bf1}=\frac12 G_1{\bf1}$, we have
\begin{eqnarray*}\label{L1x}
L_1 x{\bf1}=L_1G_0 {\bf1}=G_0L_1 {\bf1}+\frac12 G_1 {\bf1}=\lambda G_0(x^2+\alpha){\bf1}+\frac12 \lambda x{\bf1}=\lambda(x^3+(\frac12+\alpha)x){\bf1}.
\end{eqnarray*}
Assume that $G_m{\bf1}=\lambda^m x{\bf1}$ holds for $m=k\geq 1$.
Then
\begin{eqnarray*}\label{L1Gn}
    (\frac12-k)G_{k+1} {\bf1}\!\!\!&=\!\!\!&L_1G_k {\bf1}-G_kL_1{\bf1}\\
    \!\!\!&=\!\!\!&\lambda^kL_1 x{\bf1}-\lambda G_k (x^2+\alpha){\bf1}\nonumber\\
    \!\!\!&=\!\!\!&\lambda^{k+1}(x^3+(\frac12+\alpha)x){\bf1}-\lambda (x^2+k+\alpha)G_k {\bf1}\nonumber\\
    \!\!\!&=\!\!\!&\lambda^{k+1}(x^3+(\frac12+\alpha)x-x(x^2+k+\alpha)){\bf1}\\
    \!\!\!&=\!\!\!&(\frac12-k)\lambda^{k+1}x{\bf1},
\end{eqnarray*}
which implies that $G_m{\bf1}=\lambda^m x{\bf1}$ holds for all $m\in\N$. Similar arguments yield that $G_m{\bf1}=\lambda^m x{\bf1}$ holds for $m<0,m\in\Z$.
\end{proof}

We are now in the position to present the main result of this section, which gives a complete classification of free $U(\hh)$-modules of rank $1$ over the  Ramond algebra.
\begin{thm}\label{thm-3}
Let $\R$ be the super-Virasoro algebra of Ramond type. Let $M$ be
an $\R$-module such that the restriction of $M$ as a
$U(\hh)$-module is free of rank $1$. Then up to a parity, $M\cong
\Omega_{\R}(\lambda,\alpha)$ for some $\lambda\in\C^*$ and
$\alpha\in\C$ with the $\R$-module structure defined as in
(\ref{module1})-(\ref{module4}).
\end{thm}

\begin{proof}
For any $f(x^2)\in\mathbb{C}[x^2]$, by Lemma \ref{lemma-GLm} (2) and Lemma \ref{lemma-Gm}, we have
\begin{eqnarray}\label{Gmf2}
G_m f(x^2){\bf1} =f(x^2+m)G_m {\bf1}=\lambda^mxf(x^2+m){\bf1}.
\end{eqnarray}
Furthermore,
\begin{eqnarray}\label{Lmxf2}
L_m x f(x^2){\bf1} \!\!\!&=\!\!\!&L_m G_0 f(x^2){\bf1}\\
\!\!\!&=\!\!\!&G_0L_m f(x^2){\bf1}+\frac{m}{2}G_m f(x^2){\bf1}\nonumber\\
\!\!\!&=\!\!\!&\lambda^mG_0 ((x^2+m\alpha)f(x^2+m)){\bf1}+\frac{m}{2}\lambda^m xf(x^2+m){\bf1} \nonumber\\
\!\!\!&=\!\!\!&\lambda^m x
(x^2+m\alpha+\frac{m}{2})f(x^2+m){\bf1}.\nonumber
\end{eqnarray}
Assume that $G_m x{\bf1}=h_m(x^2){\bf1}\in\C[x^2]{\bf1}$. Since $G_mG_m x{\bf1}=L_{2m}x{\bf1}$, we have
$$G_m h_m(x^2){\bf1}=\lambda^mxh_m(x^2+m){\bf1}=\lambda^{2m}x(x^2+2m\alpha+m){\bf1},$$
which implies that $h_m(x^2+m)=\lambda^{m}(x^2+2m\alpha+m)$. Hence we have
$$G_m x{\bf1}=\lambda^{m}(x^2+2m\alpha){\bf1}.$$
Then it follows from Lemma \ref{lemma-GLm} (1) that
\begin{eqnarray}\label{Gmxf2}
G_m xf(x^2){\bf1}=\lambda^m (x^2+2m\alpha) f(x^2+m){\bf1}.
\end{eqnarray}
Consequently,  (\ref{Lmf2}), (\ref{Gmf2}), (\ref{Lmxf2}) and (\ref{Gmxf2}) imply that $M\cong \Omega_{\R}(\lambda,\alpha)$ as $\R$-modules.
\end{proof}

\section{Classification of free $U(\mathfrak{H})$-modules of rank $2$ over $\NS$}

Recall that the Neveu-Schwarz  algebra $\NS$ has a $1$-dimensional
Cartan subalgebra $\mathfrak{H}=\C L_0$, which lies in the even
part. Thus $U(\mathfrak{H})=\C[L_0]$. Since the Lie superalgebra
$\NS$ is generated by odd elements $G_{\frac{1}{2}+r}$, $r\in\mathbb{Z}$, $\NS$ does not
have non-trivial modules which are pure even or pure odd. Consequently, free
$U(\mathfrak{H})$-modules of rank $1$ do not exist.

In the remaining of this section, we will classify the free $U(\mathfrak{H})$-modules of rank $2$ over the Neveu-Schwarz algebra $\NS$.

Let $M=M_{\bar 0}\oplus M_{\bar 1}$ be an $\NS$-module such that it is free of rank $2$ as a $U(\mathfrak{H})$-module with two homogeneous basis elements $v$ and $w$. If the parities of $v$ and $w$ are the same, then $G_{\pm\frac{1}{2}}v=G_{\pm\frac{1}{2}}w=0$. Hence,
$$L_{\pm1}v=G_{\pm\frac{1}{2}}^2v=G_{\pm\frac{1}{2}}^2w=0,\quad  L_0v=\frac12[L_1,L_{-1}]v=0,\quad L_0w=\frac12[L_1,L_{-1}]w=0,$$
a contradiction. Consequently, $v$ and $w$ have different parities. Set $v={\bf 1}_{\bar 0}\in M_{\bar 0}$ and $w={\bf 1}_{\bar 1}\in M_{\bar 1}$. Then as a vector space, $M_{\bar 0}=\C[x]{\bf 1}_{\bar 0}$ and $M_{\bar 1}=\C[y]{\bf 1}_{\bar 1}$.

Since the even part of $\NS$ is isomorphic to the Witt algebra $\W$, we can naturally regard both $M_{\bar 0}$ and $M_{\bar 1}$ as $\W$-modules which are free of rank $1$ as $\C[L_0]$-modules. By Lemma \ref{Witt-module}, there exit  $\lambda,\mu\in\C^*$, $\alpha,\beta\in\C$ such that
\begin{eqnarray}\label{NS-Lm-1}
L_m h(x){\bf1}_{\bar 0}=\lambda^m(x+m\alpha)h(x+m){\bf1}_{\bar 0}
\end{eqnarray}
for all $h(x)\in\C[x], m\in\Z$ and
\begin{eqnarray}\label{NS-Lm-2}
L_m h(y){\bf1}_{\bar 1}=\mu^m(y+m\beta)h(y+m){\bf1}_{\bar 1}
\end{eqnarray}
for all $h(y)\in\C[y], m\in\Z$. We further need the following two preliminary results for later use.

\begin{lem}\label{key lem1 for NS}
Keep notations as above. Then $\mu=\lambda$ and there exists $c\in\C^*$ such that one of the following two cases occurs.
\begin{itemize}
\item[(i)]  $\beta=\alpha+\frac{1}{2}$, $G_{\frac{1}{2}}{\bf 1}_{\bar 0}=c{\bf 1}_{\bar 1}$ and $G_{\frac{1}{2}}{\bf 1}_{\bar 1}=\frac{1}{c}\lambda(x+\alpha){\bf 1}_{\bar 0}$.
\item[(ii)] $\beta=\alpha-\frac{1}{2}$,  $G_{\frac{1}{2}}{\bf 1}_{\bar 0}=\frac{1}{c}\lambda(y+\alpha-\frac{1}{2}){\bf 1}_{\bar 1}$ and $G_{\frac{1}{2}}{\bf 1}_{\bar 1}=c{\bf 1}_{\bar 0}$.
\end{itemize}
\end{lem}

\begin{proof}
Assume $G_{\frac{1}{2}}{\bf 1}_{\bar 0}=f(y){\bf 1}_{\bar 1},
G_{\frac{1}{2}}{\bf 1}_{\bar 1}=g(x){\bf 1}_{\bar 0}$. From
$[G_{\frac12},G_{\frac12}]{\bf 1}_{\bar 0}=2L_1{\bf 1}_{\bar 0}$
and $[L_0,G_{\frac12}]{\bf 1}_{\bar 1}=-\frac12G_{\frac12}{\bf
1}_{\bar 1}$, we have
\begin{equation*}\label{NS-eq-1}
G_{\frac{1}{2}}^2{\bf 1}_{\bar 0}=G_{\frac{1}{2}}f(y){\bf 1}_{\bar 1}=G_{\frac{1}{2}}f(L_0){\bf 1}_{\bar 1}=f(L_0+\frac{1}{2})G_{\frac{1}{2}}{\bf 1}_{\bar 1}=f(x+\frac{1}{2})g(x){\bf 1}_{\bar 0},
\end{equation*}
and
\begin{equation*}\label{NS-eq-2}
 G_{\frac{1}{2}}^2{\bf 1}_{\bar 0}=L_1{\bf 1}_{\bar 0}=\lambda(x+\alpha){\bf 1}_{\bar
 0},
\end{equation*}
which imply $f(x+\frac{1}{2})g(x)=\lambda(x+\alpha)$. Hence,
$f(x)=c, g(x)=\frac{1}{c}\lambda(x+\alpha)$, or
$f(x)=\frac{1}{c}\lambda(x+\alpha-\frac{1}{2}), g(x)=c$ for some
$c\in\C^*$.

Similarly, from
\begin{equation*}\label{eq-1}
G_{\frac{1}{2}}^2{\bf 1}_{\bar 1}=G_{\frac{1}{2}}g(x){\bf 1}_{\bar 0}=G_{\frac{1}{2}}g(L_0){\bf 1}_{\bar 0}=g(L_0+\frac{1}{2})G_{\frac{1}{2}}{\bf 1}_{\bar 0}=g(y+\frac{1}{2})f(y){\bf 1}_{\bar 1},
\end{equation*}
and
\begin{equation*}\label{eq-2}
G_{\frac{1}{2}}^2{\bf 1}_{\bar 1}=L_1{\bf 1}_{\bar
1}=\mu(y+\beta){\bf 1}_{\bar 1},
\end{equation*}
we know that $g(y+\frac{1}{2})f(y)=\mu(y+\beta)$.

(i) If $f(x)=c, g(x)=\frac{1}{c}\lambda(x+\alpha)$, then $g(y+\frac{1}{2})f(y)=\lambda(y+\alpha+\frac{1}{2})$. This implies that $\mu=\lambda$, and $\beta=\alpha+\frac{1}{2}$.

(ii) If $f(x)=\frac{1}{c}\lambda(x+\alpha-\frac{1}{2})$, $g(x)=c$, then $g(y+\frac{1}{2})f(y)=\lambda(y+\alpha-\frac{1}{2})$. This implies that $\mu=\lambda$, and $\beta=\alpha-\frac{1}{2}$.

We complete the proof.
\end{proof}

Due to Lemma \ref{key lem1 for NS}, up to a parity, we can assume $\beta=\alpha+\frac{1}{2}$, $G_{\frac{1}{2}}{\bf 1}_{\bar 0}={\bf 1}_{\bar 1}$ and $G_{\frac{1}{2}}{\bf 1}_{\bar 1}=\lambda(x+\alpha){\bf 1}_{\bar 0}$ without loss of generality.

\begin{lem}\label{key lem2 for NS}
For any $r\in\frac{1}{2}+\Z$, $h(x)\in\C[x], h(y)\in\C[y]$ we have
\begin{itemize}
\item[\rm(1)] $G_rh(x){\bf 1}_{\bar
0}=\lambda^{r-\frac{1}{2}}h(y+r){\bf 1}_{\bar 1}$.
\item[\rm(2)]
$G_rh(y){\bf 1}_{\bar
1}=\lambda^{r+\frac{1}{2}}(x+2r\alpha)h(x+r){\bf 1}_{\bar 0}$.
\end{itemize}
\end{lem}

\begin{proof}
We first show that part (1) and part (2) hold for $r=\frac{1}{2}$.
Using $G_{\frac12}L_0=(L_0+\frac12)G_{\frac12}$ and Lemma \ref{key
lem1 for NS},  we get
\begin{eqnarray}\label{NS-G0-1/2}
&& G_{\frac{1}{2}}h(x){\bf 1}_{\bar 0}=G_{\frac{1}{2}}h(L_0){\bf
1}_{\bar 0}=h(L_0+\frac{1}{2})G_{\frac{1}{2}}{\bf 1}_{\bar
0}=h(y+\frac{1}{2}){\bf 1}_{\bar 1}, \\
&& G_{\frac{1}{2}}h(y){\bf 1}_{\bar 1}=G_{\frac{1}{2}}h(L_0){\bf
1}_{\bar 1}=h(L_0+\frac{1}{2})G_{\frac{1}{2}}{\bf 1}_{\bar
1}=\lambda(x+\alpha)h(x+\frac{1}{2}){\bf 1}_{\bar
0}.\label{NS-G1-1/2}
\end{eqnarray}

Now for any $r\in\frac{1}{2}+\Z$ and $r\neq\frac{3}{2}$, by (\ref{NS-Lm-1}), (\ref{NS-Lm-2}), Lemma \ref{key lem1 for NS}, (\ref{NS-G0-1/2}) and  (\ref{NS-G1-1/2}), we have
\begin{eqnarray*}
&&(\frac{1}{2}r-\frac{3}{4})G_rh(x){\bf 1}_{\bar 0}\\
\!\!\!&=\!\!\!&[L_{r-\frac{1}{2}}, G_{\frac{1}{2}}]h(x){\bf 1}_{\bar 0}\\
\!\!\!&=\!\!\!&L_{r-\frac{1}{2}}G_{\frac{1}{2}}h(x){\bf 1}_{\bar 0}-G_{\frac{1}{2}}L_{r-\frac{1}{2}}h(x){\bf 1}_{\bar 0}\\
\!\!\!&=\!\!\!&L_{r-\frac{1}{2}}h(y+\frac{1}{2}){\bf 1}_{\bar 1}-G_{\frac{1}{2}}\lambda^{r-\frac{1}{2}}(x+(r-\frac{1}{2})\alpha)h(x+r-\frac{1}{2}){\bf 1}_{\bar 0}\\
\!\!\!&=\!\!\!&\lambda^{r-\frac{1}{2}}(y+(r-\frac{1}{2})(\alpha+\frac{1}{2}))h(y+r){\bf 1}_{\bar 1}-\lambda^{r-\frac{1}{2}}(y+\frac{1}{2}+(r-\frac{1}{2})\alpha)h(y+r){\bf 1}_{\bar 1}\\
\!\!\!&=\!\!\!&(\frac{1}{2}r-\frac{3}{4})\lambda^{r-\frac{1}{2}}h(y+r){\bf
1}_{\bar 1}
\end{eqnarray*}
and
\begin{eqnarray*}
&&(\frac{1}{2}r-\frac{3}{4})G_rh(y){\bf 1}_{\bar 1}\\
\!\!\!&=\!\!\!&[L_{r-\frac{1}{2}}, G_{\frac{1}{2}}]h(y){\bf 1}_{\bar 1}\\
\!\!\!&=\!\!\!&L_{r-\frac{1}{2}}G_{\frac{1}{2}}h(y){\bf 1}_{\bar 1}-G_{\frac{1}{2}}L_{r-\frac{1}{2}}h(y){\bf 1}_{\bar 1}\\
\!\!\!&=\!\!\!&L_{r-\frac{1}{2}}\lambda(x+\alpha)h(x+\frac{1}{2}){\bf 1}_{\bar 0}-G_{\frac{1}{2}}\lambda^{r-\frac{1}{2}}(y+(r-\frac{1}{2})(\alpha+\frac{1P}{2}))h(y+r-\frac{1}{2}){\bf 1}_{\bar 1}\\
\!\!\!&=\!\!\!&\lambda^{r-\frac{1}{2}}(x+(r-\frac{1}{2})\alpha)\lambda(x+r-\frac{1}{2}+\alpha)h(x+r){\bf 1}_{\bar 0}\\
&&-\lambda(x+\alpha)\lambda^{r-\frac{1}{2}}(x+\frac{1}{2}+(r-\frac{1}{2})(\alpha+\frac{1}{2}))h(x+r){\bf 1}_{\bar 0}\\
\!\!\!&=\!\!\!&(\frac{1}{2}r-\frac{3}{4})\lambda^{r+\frac{1}{2}}(x+2r\alpha)h(x+r){\bf
1}_{\bar 0}.
\end{eqnarray*}
Hence, part (1) and part (2) hold for any $r\in\frac{1}{2}+\Z$ and
$r\neq\frac{3}{2}$.

Furthermore, we have
\begin{eqnarray*}
\frac{3}{2}G_{\frac{3}{2}}h(x){\bf 1}_{\bar 0}\!\!\!&=\!\!\!&[L_{2}, G_{-\frac{1}{2}}]h(x){\bf 1}_{\bar 0}\\
\!\!\!&=\!\!\!&L_{2}G_{-\frac{1}{2}}h(x){\bf 1}_{\bar 0}-G_{-\frac{1}{2}}L_{2}h(x){\bf 1}_{\bar 0}\\
\!\!\!&=\!\!\!&L_{2}\lambda^{-1}h(y-\frac{1}{2}){\bf 1}_{\bar 1}-G_{-\frac{1}{2}}\lambda^{2}(x+2\alpha)h(x+2){\bf 1}_{\bar 0}\\
\!\!\!&=\!\!\!&\lambda(y+2(\alpha+\frac{1}{2}))h(y+\frac{3}{2}){\bf 1}_{\bar 1}-\lambda(y-\frac{1}{2}+2\alpha)h(y+\frac{3}{2}){\bf 1}_{\bar 1}\\
\!\!\!&=\!\!\!&\frac{3}{2}\lambda h(y+\frac{3}{2}){\bf 1}_{\bar 1}
\end{eqnarray*}
and
\begin{eqnarray*}
\frac{3}{2}G_{\frac{3}{2}}h(y){\bf 1}_{\bar 1}\!\!\!&=\!\!\!&[L_{2}, G_{-\frac{1}{2}}]h(y){\bf 1}_{\bar 1}\\
\!\!\!&=\!\!\!&L_2G_{-\frac{1}{2}}h(y){\bf 1}_{\bar 1}-G_{-\frac{1}{2}}L_{2}h(y){\bf 1}_{\bar 1}\\
\!\!\!&=\!\!\!&L_{2}(x-\alpha)h(x-\frac{1}{2}){\bf 1}_{\bar 0}-G_{-\frac{1}{2}}\lambda^{2}(y+2(\alpha+\frac{1}{2}))h(y+2){\bf 1}_{\bar 1}\\
\!\!\!&=\!\!\!&\lambda^{2}(x+2\alpha)(x+2-\alpha)h(x+\frac{3}{2}){\bf 1}_{\bar 0}-\lambda^2(x-\alpha)(x-\frac{1}{2}+2(\alpha+\frac{1}{2}))h(x+\frac{3}{2}){\bf 1}_{\bar 0}\\
\!\!\!&=\!\!\!&\frac{3}{2}\lambda^2(x+3\alpha)h(x+\frac{3}{2}){\bf
1}_{\bar 0}.
\end{eqnarray*}
Hence,  part (1) and part (2) also hold for $r=\frac{3}{2}$. We
complete the proof of Lemma \ref{key lem2 for NS}.
\end{proof}

We are now in the position to present the following main result of this section, the proof of which follows from (\ref{NS-Lm-1}), (\ref{NS-Lm-2}), Lemma \ref{key lem1 for NS} and  Lemma \ref{key lem2 for NS}. It gives a complete classification of free $U(\mathfrak{H})$-module of rank $2$ over the super-Virasoro algebra of Neveu-Schwarz type.
\begin{thm}\label{thm-NS}
\label{thm} Let $\NS$ be the  Neveu-Schwarz algebra. Let $M$ be  an $\NS$-module such that the restriction of $M$ as a $U(\mathfrak{H})$-module is free of rank $2$. Then up to a parity, $M\cong \Omega_{\NS}(\lambda,\alpha)$ for some $\lambda\in\C^*$ and $\alpha\in\C$ with the $\NS$-module structure defined as in (\ref{ns-module1})-(\ref{ns-module4}).
\end{thm}

As a direct consequence of Proposition \ref{Connection of modules over R and NS}, Theorem \ref{thm-3} and Theorem \ref{thm-NS}, we have
\begin{coro}
The category of free $U(\hh)$-modules of rank $1$ over the Ramond algebra $\R$ is equivalent to the category of free $U(\mathfrak{H})$-modules of rank $2$ over the  Neveu-Schwarz algebra $\NS$.
\end{coro}

\end{document}